\documentclass[a4paper, 11pt, reqno]{amsart}
\usepackage{amssymb, amsmath, latexsym, graphics, graphicx}
\newtheorem{theorem}{Theorem}[section]
\newtheorem{lemma}[theorem]{Lemma}
\newtheorem{corollary}[theorem]{Corollary}
\newtheorem{example}[theorem]{Example}
\newtheorem{proposition}[theorem]{Proposition}

\begin{document}
\title{Face monoid actions and tropical hyperplane arrangements}

\maketitle

\begin{center}
MARIANNE JOHNSON\footnote{Email \texttt{Marianne.Johnson@maths.manchester.ac.uk}.}  and MARK KAMBITES\footnote{Email \texttt{Mark.Kambites@manchester.ac.uk}.}

\medskip

School of Mathematics, \ University of Manchester, \\
Manchester M13 9PL, \ England.

\date{\today}

\thanks{}
\end{center}
\numberwithin{equation}{section}

\begin{abstract}
We study the combinatorics of tropical hyperplane arrangements, and
their relationship to (classical) hyperplane face monoids.
We show that the \emph{refinement} operation on the faces of a tropical
hyperplane arrangement, introduced by Ardila and Develin \cite{Ard09} in
their definition of a \emph{tropical oriented matroid}, induces an action
of the hyperplane face monoid of the classical braid arrangement on the
arrangement, and hence on a number of interesting related structures.
Along the way, we introduce a new characterization of the \textit{types}
(in the sense of Develin and Sturmfels \cite{DevStu04}) of
points with respect to a tropical hyperplane arrangement, in terms of
\textit{partial bijections} which attain permanents of submatrices
of a matrix which naturally encodes the arrangement.
\end{abstract}

\renewcommand{\thefootnote}{\fnsymbol{footnote}}
\renewcommand{\thefootnote}{\arabic{footnote}}

\section{Introduction}
The set of faces of a (central) hyperplane arrangement, and more generally the set of covectors of an oriented matroid, provide an important source of examples in the class of finite monoids known as \emph{left regular bands}. Such monoids have been the focus of much research interest, following the celebrated work of Bidigare, Hanlon and Rockmore~\cite{BHR99} who showed that a number of well-known Markov chains, including the Tsetlin library and the riffle shuffle, are random walks on the faces of a hyperplane arrangement, and that the representation theory of the hyperplane face monoid could be used to analyze these Markov chains. This observation created substantial interest in random walks on hyperplane face monoids and related semigroups~\cite{AthDia10, Bjo08, BroDia98, ReSaWe14} and the theory of left regular band walks more generally \cite{Bro00, Bro04, Sal12}, and in various aspects of the representation theory of hyperplane face monoids \cite{Sal09, Sch06}, general finite left regular bands \cite{Sal07} and $\mathcal{R}$-trivial monoids~\cite{ASST15, BBBS11}, as well as the study of several related (quasi)varieties of semigroups and monoids~\cite{AMSV09, MSSte14}. The representation theory of hyperplane face monoids is also closely connected to the Solomon descent algebra (see \cite{Bro04, Sch06, Sol76, Tits76} for further details) and hyperplane face monoids themselves also arise in connection with combinatorial Hopf algebras \cite{AguMah06}.

The study of the tropical complex generated by the columns of a real
$n\times d$ matrix $M$ (see Section \ref{sect:tropfaces}) was initiated by
Develin and Sturmfels \cite{DevStu04}. It has been observed
\cite{Ard09} that such a tropical complex can be viewed as the
analogue of the abstract simplicial complex of faces resulting from a real
hyperplane arrangement, by associating to each column of $M$ a min-plus
linear form specifying a min-plus hyperplane (this will be explained in
detail in Section 3).

In the case of real hyperplane arrangements, the corresponding face poset is a monoid with a natural geometrically (or combinatorially) defined product. There is no obvious analogous monoid structure on the face poset of a tropical hyperplane arrangement, but we shall show instead that the hyperplane face monoid $\mathcal{P}_n$ of the classical \emph{braid arrangement} (see Section \ref{sect:braid}) acts upon the set of faces determined by $d$ min-plus tropical hyperplanes in $\mathbb{R}^n$.

The faces of the tropical complex generated by $M$ are labelled by \emph{types} which are the tropical analogue of the \textit{sign sequences}
of a real hyperplane arrangement. We identify the collection of all types
with the face poset $\mathcal{F}(M)$ of the tropical complex --- this is the
analogue of the oriented matroid of covectors of a real hyperplane
arrangement and hence provided the inspiration for the notion of a
\emph{tropical oriented matroid} introduced by Ardila and
Develin~\cite{Ard09}. We shall show that their definition induced a
natural right action of the monoid $\mathcal{P}_n$ on the tropical face poset $\mathcal{F}(M)$ and there is a geometric interpretation of this action when viewing the elements of $\mathcal{F}(M)$ as faces. By viewing $\mathcal{F}(M)$ instead as a poset of Boolean matrices (the types), this action is just the restriction of the action of $\mathcal{P}_n$ on the set of all $n \times d$ Boolean matrices described in Section \ref{sect:action} below.

Without the geometric intuition of the action of $\mathcal{P}_n$ it is not so easy to see why each element of $\mathcal{P}_n$ should act upon $\mathcal{F}(M)$; indeed, one needs to be able to determine whether a given $n \times d$ Boolean matrix is in fact a type labelling a face of $\mathcal{F}(M)$. In this paper we
give a purely combinatorial characterization of the types of $\mathcal{F}(M)$, and use this to give a combinatorial argument showing that $\mathcal{P}_n$ acts on the set of all types. To this end, we introduce another combinatorial object $\mathcal{P}(M)$ called the \emph{permanent structure} of $M$, and show that this provides an equivalent encoding of the combinatorial data of the tropical complex generated by $M$ (Corollary \ref{cor:sameinfo}). This may be of independent interest.

The paper is organised as follows. In Section 2 we recall some facts about
real hyperplane arrangements and the hyperplane face monoid $\mathcal{P}_n$, including its
natural action on the power set of $\lbrace 1, \dots, n \rbrace$.
In Section 3 we
recall the necessary background on tropical hyperplane arrangements. In
Section 4 we begin by making a distinction between Boolean matrices whose
corresponding order theoretic conditions can be \emph{satisfied} by a
point and those which cannot.  We show that a Boolean matrix is
\emph{satisfiable} if and only if every partial bijection contained within
it (as a submatrix) is satisfiable. In Section 5 we show that the
satisfiable partial bijections are precisely the max-plus
permanent-attaining ones (and hence a matrix is satisfiable if and only if
every partial bijection contained within it attains the permanent). In
Section 6 we give a precise characterization of those matrices which
correspond to types, in terms of permanent-attaining partial
bijections. As an application, we obtain a combinatorial understanding of
the action of
$\mathcal{P}_n$ on $\mathcal{F}(M)$.

\subsection*{Notation}
We denote by $[n]$ the set $\lbrace 1, \dots, n \rbrace$. A key object
of study will be the set of $d$-tuples of subsets of $[n]$. It is
convenient to identify each such tuple $(S_1, \dots, S_d)$ with the
$n \times d$ zero-one (Boolean) matrix $M$ such that $M_{ij} = 1$ if and
only if $i \in S_j$. We write $\mathbb{B}^{n \times d}$ for the set of
all $n \times d$ Boolean matrices. We write $M_{i\star}$ and $M_{\star j}$
for the $i$th row and $j$th column respectively of a (Boolean or other)
matrix $M$. For a Boolean matrix $M$ we will also view each row $M_{i\star}$ (respectively, column $M_{j\star}$) as the corresponding subset of $[d]$ (respectively $[n]$) when convenient.

Notice that the \textit{dual} $n$-tuple $(T_1, \dots, T_n)$
of subsets of $[d]$, where $i \in T_j$ if and only if $j \in S_i$ corresponds
to the transpose matrix, so with our identification, we have the
natural notation $(S_1, \dots, S_d)^{\mathsf T} = (T_1, \dots, T_n)$.

There is an obvious partial order on $\mathbb{B}^{n \times d}$ given by
setting $M \preceq N$ if and only if $M_{ij} \leq N_{ij}$ for all $i$ and
$j$. Equivalently in terms of tuples, $(S_1, \dots, S_d) \preceq
(T_1, \dots, T_d)$ if and only if $S_i \subseteq T_i$
for each $i$.

By a \emph{partial bijection} $\sigma : [d] \dashrightarrow [n]$ we mean
a bijection $\sigma: J \rightarrow I$ between subsets $J \subseteq [d]$
and $I \subseteq [n]$. We may identify such a partial bijection with the
matrix $\Sigma \in \mathbb{B}^{n\times d}$ defined by $\Sigma_{i, j} =1$
if and only if $i=\sigma(j)$; thus the partial bijections correspond to
the elements of $\mathbb{B}^{n\times d}$ having at most one non-zero entry
in each row and in each column.

\section{Face monoid actions}
In this section we recall some necessary background on hyperplane face monoids (introduced in Section \ref{sect:faces}).
Our main interest is in the hyperplane face monoid of the braid arrangement (Section \ref{sect:braid}) and a particular
action of this monoid on the set of Boolean matrices (Section \ref{sect:action}).

\subsection{The face monoid of a hyperplane arrangement} \label{sect:faces}
Consider a finite collection $\mathcal{A} = \{\mathcal{H}_1^0, \ldots, \mathcal{H}_d^0\}$ of (classical) \emph{linear hyperplanes} containing the origin in $\mathbb{R}^n$, each of which may be written in the form
$$\mathcal{H}_i^0 = \{ x \in \mathbb{R}^n: a_i \cdot x = 0\}$$
for some fixed choice of $a_i \in \mathbb{R}^n$. The complement of a hyperplane arrangement in $\mathbb{R}^n$ is a collection of open subsets of $\mathbb{R}^n$ called \emph{chambers}. More generally, a \emph{face} of the arrangement is defined to be a non-empty set which is an intersection of the form $\bigcap_{\mathcal{H}_i \in \mathcal{A}} \mathcal{H}_{i}^{\sigma_i},$
where $\sigma_i \in \{-,0,+\}$ and
\begin{eqnarray*}
\mathcal{H}^+_{i} = \{x\in \mathbb{R}^n: a_i\cdot x>0\},  \;\;\; \mathcal{H}^-_{i} = \{x\in \mathbb{R}^n: a_i \cdot x <0\}.
\end{eqnarray*}
Note that the sequence of signs $(\sigma_i)$ therefore specifies on which `side' of each hyperplane $\mathcal{H}_i^0 \in \mathcal{A}$ the given face lies. For each $x \in \mathbb{R}^n$ and each hyperplane $\mathcal{H}_i^0$, we let $\sigma_i(x) \in \{-, 0, +\}$ denote the sign indicating the position of the point $x$ relative to the hyperplane $\mathcal{H}_i^0$, so that $x \in \mathcal{H}_i^{\sigma_i(x)}$. It is then clear that the map $\sigma: \mathbb{R}^n \rightarrow \{-,0,+\}^{|\mathcal{A}|}$ defined by $\sigma(x) = (\sigma_1(x), \ldots, \sigma_d(x))$ is constant on faces and that distinct faces have distinct images under $\sigma$. Thus we may identify the set of all faces $\mathcal{F}$ with the set of sign sequences $\sigma(\mathbb{R}^n)$. Under this identification, the chambers are those elements of the image in which each component is non-zero, and the face $Z$ corresponding to the common intersection of all hyperplanes in $\mathcal{A}$ is identified with the all-zero sequence.

It is well known (see, for example, ~\cite{BroDia98}) that the set of all faces $\mathcal{F}$ of any finite linear hyperplane arrangement can be endowed with the structure of a monoid, whose identity element is $Z$, and whose product can be defined geometrically as follows. For faces $F$ and $G$ of a given hyperplane arrangement, choose points $f \in F$ and $g\in G$. The product $F*G$ is the very first face one encounters after leaving point $f$ when walking in a straight line towards $g$ (which could be the face $F$ itself); this definition does not depend upon the particular points selected, and it can be shown that this product is associative. The resulting finite monoid is called the \emph{hyperplane face monoid} of the arrangement and it is well-known and straightforward to check that each such monoid is a \emph{left regular band} (that is, $F*F=F$ and $F*G*F=F*G$ for all $F,G \in \mathcal{F}$) and that each chamber is a \emph{left zero} of this monoid (that is, for any chamber $C \in \mathcal{F}$, we have $C* F = C$ for all $F \in \mathcal{F}$). We refer the reader to \cite{Bro00} and \cite{Sal09} for general reference on properties of hyperplane face monoids.

The set of signs $L=\{-,0,+\}$ can be considered as a left regular band consisting of an identity element $0$ and two left zeroes, $-$ and $+$. It turns out that $\sigma$ induces a monoid embedding of $\mathcal{F}$ in the direct power $L^{|\mathcal{A}|}$ (see~\cite{Bro04} for example). In other words, upon identifying $\mathcal{F}$ with the set of possible sign sequences $\sigma(\mathbb{R}^n)$, for each $\mathcal{H}_i \in \mathcal{A}$ and all $F, G \in \mathcal{F}$ we have
$$(F*G)_i = \begin{cases}
F_i & \mbox{ if }F_i\neq 0\\
G_i & \mbox{  if } F_i =0.
\end{cases}$$
In fact, as observed in~\cite[\S 3E]{BroDia98}, the elements of $\mathcal{F} = \sigma(\mathbb{R}^n)\subseteq L^{|\mathcal{A}|}$ satisfy the covector axioms of an \emph{oriented matroid} (we refer the reader to ~\cite{BLSWZ99} for general reference on oriented matroids and also \cite[Section 2.3.1]{MSSte15} for a comparison between the terminology of hyperplane face monoids and oriented matroids). In \cite[Definition 3.5]{Ard09} Ardila and Develin suggested a definition of \emph{tropical oriented matroid}, based on their geometric understanding of both tropical hyperplane arrangements and the oriented matroids arising from (classical) hyperplane arrangements. We shall show that
their definition induces an \emph{action} of a certain classical hyperplane face monoid on the faces of a tropical hyperplane arrangement.

\subsection{The face monoid of the braid arrangement}\label{sect:braid}
In the case of Coxeter arrangements, the geometric description of the product of faces $F*G$ described above can be viewed as the Tits projection of $G$ on $F$, denoted ${\rm proj}_F G$ ~\cite{Tits76}. One well-studied example is the linear hyperplane arrangement
$$\mathcal{H}^0_{i,j} = \{x\in \mathbb{R}^n: x_i=x_j\} \mbox { for } 1\leq i <j \leq n,$$
known as the \emph{braid arrangement} $\mathcal{B}_{n}$, consisting of $\binom{n}{2}$ hyperplanes in $\mathbb{R}^n$ intersecting in the one-dimensional real vector space $\mathbb{R}(1,\ldots, 1)$. The corresponding hyperplane face monoid has received considerable attention in connection with shuffling schemes, random walks and the Solomon descent algebra of the
symmetric group \cite{BHR99,BroDia98, Bro00, Bro04,Sch06}.

It is clear that the braid arrangement yields $n!$ chambers of the form
$$C_{\tau}  = \{ (x_1, \ldots, x_n): x_{\tau(1)} > \cdots > x_{\tau(n)}\}.$$
for $\tau$ in the symmetric group ${\rm Sym}(n)$.
Moreover, since a face of this arrangement is defined to be a non-empty intersection of the form $\cap \mathcal{H}_{i,j}^{\sigma_{i,j}},$ where $\sigma_{i,j} \in \{-,0,+\}$,
\begin{eqnarray*}
\mathcal{H}^+_{i,j} = \{x\in \mathbb{R}^n: x_i>x_j\}, \mbox{ and } \mathcal{H}^-_{i,j} = \{x\in \mathbb{R}^n: x_i<x_j\},
\end{eqnarray*}
it is easy to see that the faces of the braid arrangement can be identified with ordered partitions of $[n]$: for example, the face $$\{x \in \mathbb{R}^7: x_1=x_3=x_4 >x_6>x_2 =x_7 >x_5\}$$ is identified with the ordered partition $(\{1,3,4\}, \{6\}, \{2,7\},  \{5\})$. Let $\mathcal{P}_n$ denote the set of all ordered set-partitions of $[n]$. Viewing faces $F$ and $G$ of the braid arrangement as ordered partitions $F=(F_1, \ldots, F_l)$ and $G=(G_1,\ldots, G_r)$ it is known (see~\cite{Bro04} or \cite{Sch06} for example) that their product $F * G \in \mathcal{P}_n$  may be written concretely as
\begin{equation}
\label{setprod}F*G = (F_1 \cap G_1, \ldots, F_1 \cap G_r, F_2 \cap G_1, \ldots,F_2 \cap G_r, \ldots, F_l \cap G_r )^\sharp,
\end{equation}
where $(\cdot)^\sharp$ denotes the operation of deleting any empty sets. It is straightforward to check that the partition $(\{1,\ldots, n\})$ acts as an identity on both sides, and that the chambers $C_{\tau}$ of the braid arrangement correspond to ordered partitions $(\{\tau(1)\}, \ldots, \{\tau(n)\})$, each of which is a left zero of the monoid.

Since $\mathcal{P}_n$ is a monoid it acts (by multiplication) on both the left and the right of $\mathcal{P}_n$ itself. It is clear that each chamber is a fixed point of the \emph{right action} of faces, since each chamber is a left zero element of the monoid. We will show that the right action of the monoid $\mathcal{P}_n$ on $\mathcal{C}$ can be extended in a natural way to give a right action on a number of interesting combinatorial structures resulting from a tropical hyperplane arrangement in $\mathbb{R}^n$.

\subsection{An action on Boolean matrices}\label{sect:action}
We begin by showing that $\mathcal{P}_n$ acts on the right of the power set of $[n]$ as follows: given any subset  $I \subseteq [n]$ and any $F \in \mathcal{P}_n$ we define
\begin{equation}
\label{setaction}I \circ F := \begin{cases}
\emptyset & \mbox{ if }I=\emptyset\\
 I\cap F_j & \mbox{  if } I \cap F_j \neq \emptyset \mbox{ and } I \cap F_k = \emptyset \mbox{ for all } k >j.
\end{cases}\end{equation}
That is, $I \circ F$ is the right-most non-empty intersection of $I$ with a
block of the partition $F$. Note that this operation is well-defined,
since if $I \neq \emptyset$, then $I$ must have non-empty intersection
with \emph{some} component of the ordered set partition $F$ of $[n]$.
\begin{proposition}\label{prop_paction}
The above operation gives a right action of the face monoid $\mathcal{P}_n$
on the power set of $[n]$. Moreover, $I \circ F \subseteq I$ for all $I \subseteq [n]$
and $F \in \mathcal{P}_n$.
\end{proposition}
\begin{proof}
Recall that the identity of the monoid $\mathcal{P}_n$ is the trivial
partition $([n])$; it is immediate from the definition that
$I \circ ([n]) = I \cap [n] = I$.

It remains to show that $I \circ (F*G) = (I\circ F) \circ G$ for any
$I \subseteq [n]$ and any two faces $F, G \in \mathcal{P}_n$. This is
clear if $I=\emptyset$, so suppose that $I$ is non-empty. By definition
\eqref{setaction}, $I \circ (F*G) = I \cap (F*G)_j$ where $j$ is the
largest index of the product providing a non-zero intersection with $I$.
Using the concrete interpretation \eqref{setprod} of the product $F*G$, we
conclude that $I \circ (F*G) = I \cap F_i\cap G_k$ where $i$ is the
maximal index of $F$ such that $I \cap F_i \cap G_s$ is non-empty for some
$s$, and $k$ is the maximal index of $G$ such that $I \cap F_i \cap G_k$
is non-empty. Since $G$ is a partition, this is equivalent to saying $I
\circ (F*G) = I \cap F_i\cap G_k$ where $i$ is the maximal index of $F$
such that $I \cap F_i$ is non-empty, and $k$ is the maximal index of $G$
such that $I \cap F_i \cap G_k$ is non-empty, which is precisely the
definition of $(I \circ F)\circ G$ obtained by applying \eqref{setaction}
twice.

That $I \circ F \subseteq I$ is immediate from the definition.
\end{proof}

The last part of Proposition~\ref{prop_paction} should not be misinterpreted as saying that the action of $\mathcal{P}_n$ is
monotonically decreasing with respect to the containment order: if $I \subseteq J \subseteq [n]$ then we
do \textit{not} necessarily have $I \circ F \subseteq J \circ F$. (What is true is that either $I \circ F \subseteq J \circ F$
or $I \circ F \cap J \circ F = \emptyset$.)

For any $d$ the action of $\mathcal{P}_n$ on the powerset of $[n]$ obviously extends
(componentwise) to an action on the set of $d$-tuples of subsets of $[n]$
and hence (through the identification described in the Introduction) on
the set $\mathbb{B}^{n \times d}$ of $n \times d$ Boolean matrices. It
follows from Proposition~\ref{prop_paction} that
$M \circ F \preceq M$ for all $M \in \mathbb{B}^{n \times d}$ and all $F \in \mathcal{P}_n$. In particular, the
action of $\mathcal{P}_n$ restricts to an action on any subset of
$\mathbb{B}^{n \times d}$ which is downward-closed under the partial order.

\section{Tropical hyperplane arrangements}
\subsection{Tropical linear forms and hyperplanes}\label{sect:trophyp}
For $a,b \in \mathbb{R}$ write $a \oplus b := \max(a,b)$, $a \boxplus b:= \min(a,b)$ and $a \otimes b:=a + b$. The operations $\oplus$ and $\otimes$ give $\mathbb{R}$ the structure of a \textit{semiring} (without zero element), called the \textit{max-plus semiring} (see \cite{But10,Cun79} for further reference). The min-plus semiring is defined similarly, and negation provides an isomorphism between these two structures, allowing one to easily interpret min-plus results in the max-plus setting and vice versa. The word \emph{tropical} is used throughout the literature to describe a number of algebraic and geometric constructions involving either the max-plus or min-plus semiring. For example, it is clear that these operations extend naturally to give notions of addition and scalar multiplication of vectors, hence giving $\mathbb{R}^n$ the structure of a (semi)module over the max-plus or min-plus semiring. A subset of $\mathbb{R}^n$ closed under  max-plus (or min-plus) operations will be referred to as a max-plus (respectively min-plus) submodule of $\mathbb{R}^n$; the finitely generated submodules have been termed max-plus (min-plus) ``tropical polytopes'' in the literature. The theory of tropical semimodules and convex sets has been developed by several authors, including Cohen, Gaubert and Quadrat~\cite{CGQ04}, and Develin and Sturmfels \cite{DevStu04}.

By analogy with the case of real hyperplanes, each $a\in \mathbb{R}^n$ can be used to define a max-plus or min-plus hyperplane, by considering the linear forms
\begin{eqnarray}
\label{maxhyp}(a_1 \otimes x_1) \oplus (a_2\otimes x_2) \oplus \cdots \oplus (a_n\otimes x_n)\\
\label{minhyp}(a_1\otimes x_1) \boxplus (a_2\otimes x_2) \boxplus \cdots \boxplus (a_n\otimes x_n),
\end{eqnarray}
in co-ordinate variables $x_1, \ldots, x_n$. The max-plus (respectively, min-plus) hyperplane is then defined to be the set of all points $x\in \mathbb{R}^n$ for which the maximum (respectively minimum) is \emph{attained twice} in expression \eqref{maxhyp} (respectively expression \eqref{minhyp}) above. In contrast to hyperplanes in usual Euclidean space (which split the containing space into two half-spaces), a tropical hyperplane divides the containing space $\mathbb{R}^n$ into $n$ tropical ``half''-spaces or \emph{sectors}, according to which co-ordinate is maximal (respectively, minimal). It is clear that the point $x=-a$ lies on the hyperplane defined by the linear form specified by $a$, and that this point is adjacent to each of the sectors, in the sense that one can move to any given sector by making some small perturbation to the co-ordinates of $a$. We therefore define a (max-plus/min-plus) tropical hyperplane with \emph{apex} $a$ to be the (max-plus/min-plus) tropical hyperplane defined by the linear form specified by $-a$. Note that this tropical hyperplane does not have a unique apex, however the apexes are precisely all tropical scalings $\lambda \otimes a$.

There is a natural partial order on $\mathbb{R}^n$ given by $x \leq y$ if and only if $x_i \leq y_i$ for all $i$. For $x, y \in \mathbb{R}^n$ we write
\begin{equation}
\label{resid}
\langle x \mid y \rangle = {\rm max} \{\lambda \in \mathbb{R}: \lambda \otimes x \leq y\} = {\rm min}_k\{y_k-x_k\}.
\end{equation}
This operation is a \emph{residuation} operator (see \cite{Blyth}) and plays an important role in tropical mathematics (see for example \cite{CGQ04}). We say that $x$ \emph{dominates} $y$ in position $i$ if $\langle x \mid y \rangle = y_i-x_i$. It is straight-forward to verify that domination is scale invariant, that is, if $x$ dominates $y$ in position $i$, then $\lambda \otimes x$ dominates $\mu \otimes y$ in position $y$ for all $\lambda, \mu \in \mathbb{R}$. The set
\begin{equation}
\label{Domset}
{\rm Dom}_i(x) = \Big\{ y \in \mathbb{R}^n: y_i-x_i = {\rm min}_k\{y_k-x_k\} \Big\}
\end{equation}
consisting of all elements of $\mathbb{R}^n$ dominated by $x$ in position $i$ can be seen to be a max-plus (and min-plus) submodule of $\mathbb{R}^n$ as well as a closed convex set when considered as a subset of Euclidean space $\mathbb{R}^n$ (see \cite[Lemma 1.1]{JKconvex} for example). If $H$ is the min-plus hyperplane with apex $a\in \mathbb{R}^n$, then the set
${\rm Dom}_i(a)$ is the closed sector consisting of points $x$ such that $x_i - a_i \leq x_k - a_k$ for all $k$.

Since domination is scale invariant, it follows that we may identify each hyperplane $H$ and the corresponding domination sets ${\rm Dom}_i(a)$ with their respective images in the (classical) quotient vector space $\mathbb{R}^n/\mathbb{R}(1,\ldots,1)$, termed \emph{tropical projective space}, as convenient. It will also be convenient (for the purpose of drawing pictures) to identify $\mathbb{R}^n/\mathbb{R}(1,\ldots,1)$ with $\mathbb{R}^{n-1}$, via the linear map
\begin{equation}
\label{eq:projection}
(v_1, \ldots, v_n) \mapsto (v_1-v_n, \ldots, v_{n-1}-v_n).
\end{equation}

\begin{figure}[h!]
\includegraphics[scale=0.5]{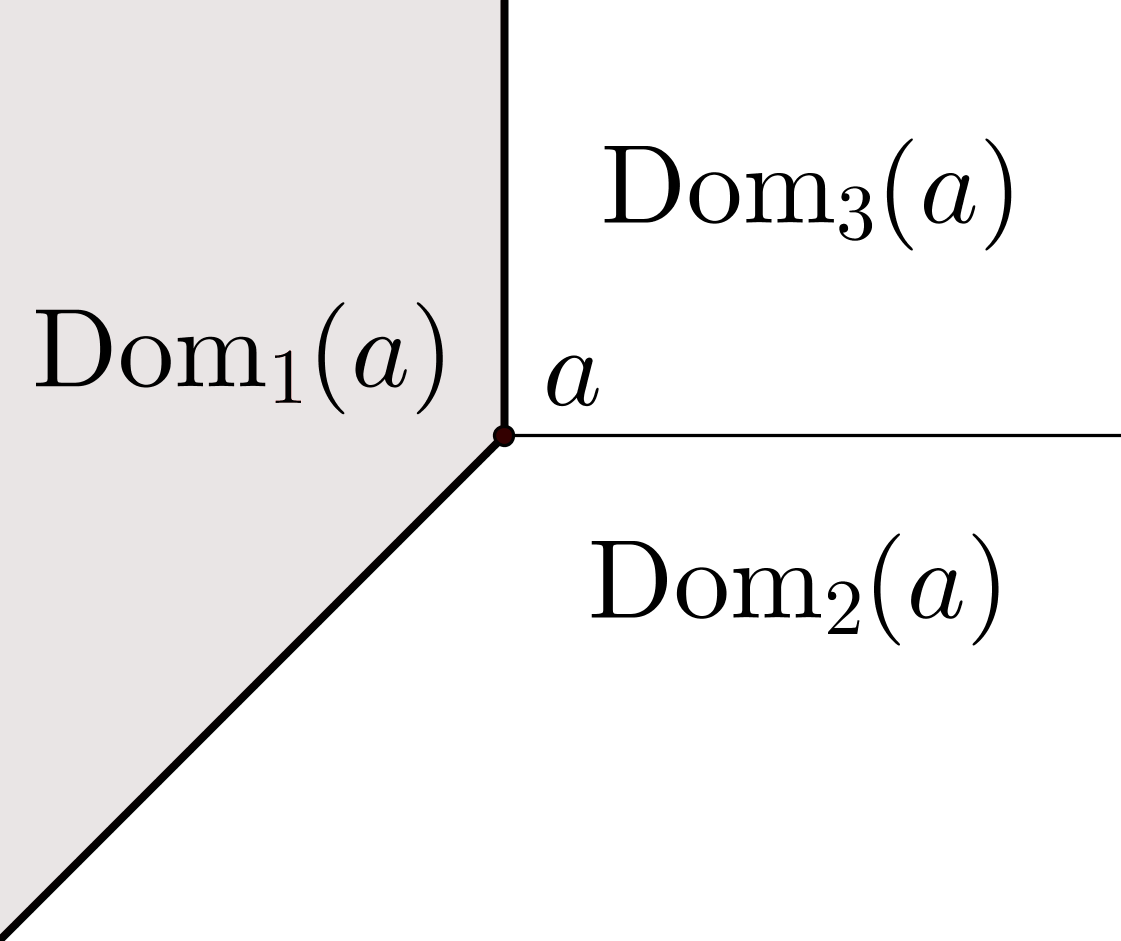}
\caption{The min-plus tropical hyperplane with apex $a \in \mathbb{R}^3$.}
\end{figure}

Figure 1 illustrates a single min-plus tropical hyperplane, viewed in tropical projective space $\mathbb{R}^3/\mathbb{R}(1,1,1)$ which is identified with $\mathbb{R}^2$ (consisting of three half rays corresponding to the three ways in which the minimum can be attained at least twice), together with a labelling of the sectors following the usual co-ordinate axis conventions. Note that by setting $a=(0,0,0)$ and extending the half-rays to give full real hyperplanes, one obtains the braid arrangement $\mathcal{B}_{3}$.

\subsection{The faces of a tropical hyperplane arrangement}\label{sect:tropfaces}
Let $M$ be a real $n \times d$ matrix.
The columns of $M$ determine a min-plus \emph{tropical
hyperplane arrangement} $\mathcal{H}_1, \ldots, \mathcal{H}_d$ in $\mathbb{R}^n$, where $\mathcal{H}_j$ is the min-plus tropical hyperplane with apex
$M_{\star j}$, or in other words, the set of all $y \in \mathbb{R}^n$ such that the minimum over $k$ of $y_k - M_{k,j}$ is
attained twice.

The \emph{type} of a point $x \in \mathbb{R}^n$ with respect to the
hyperplane arrangement specified by $M$ is an
$n \times d$ Boolean matrix (or equivalently, an $d$-tuple of subsets of
$[n]$) recording in which of the closed sectors
relative to each $\mathcal{H}_j$ the point $x$ lies. Specifically, the type of
$x$ with respect to $M$ is the Boolean matrix $T$ where $T_{ij} = 1$ if
and only  $x \in {\rm Dom}_i(M_{\star j})$.

Types may be viewed as the tropical analogue of the sign sequences associated with a real hyperplane arrangement; the
collection of all types (whilst not a matroid in the usual sense) is clearly the analogue of the oriented matroid of
covectors of a real hyperplane arrangement, and motivated the notion of \textit{tropical oriented matroids}
studied in \cite{Ard09}. Develin and Sturmfels \cite[Theorem 15]{DevStu04} have shown that the min-plus hyperplane arrangement given by (the columns of) $M$ induces a polyhedral cell complex structure on $\mathbb{R}^n$, called the \emph{tropical complex} generated by $M$, whose face poset $\mathcal{F}(M)$ is naturally labelled by the types.

Since domination is scale invariant we note that the type of $\lambda \otimes x$ is the same as the type of $x$ for all $\lambda \in \mathbb{R}$, $x \in \mathbb{R}^n$. It follows that we may identify the hyperplanes $\mathcal{H}_j$ and the faces of the resulting tropical complex with their respective images in tropical projective space $\mathbb{R}^n/\mathbb{R}(1,\ldots,1)$, and hence by \eqref{eq:projection} $\mathbb{R}^{n-1}$, when convenient. Under this latter identification the faces of the max-plus tropical polytope are precisely the bounded faces in the tropical complex \cite[Theorem 15]{DevStu04}. By the \emph{dimension} of a face we shall mean the dimension of the corresponding polyhedron in $\mathbb{R}^{n-1}$.

\begin{example}
Consider the matrix
$$M = \left(\begin{array}{cccc}
-8& 10 & 15 & 0\\
10 & 10 & 5 & -10\\
0&0&0&0
\end{array}\right)$$
The min-plus tropical hyperplane arrangement specified by the columns of $M$ is shown in Figure 2. As viewed in projective space, the resulting tropical complex has 12 cells of dimension 2 (such as the bounded cell $F$ and the unbounded cell $E$), 18 cells of dimension 1 (such as the bounded cell $G$, which is a face of both $E$ and $F$) and 7 cells of dimension 0 (such as the point marked $H$, which is a face of each of the other labelled cells). The bounded cells have been shaded, and these cells constitute the max-plus polytope generated by the columns of $M$ (see \cite[Theorem 15]{DevStu04}). We identify the cells $E,F, G$ and $H$ with their corresponding types:
\begin{eqnarray*}
E \ &=& \ \left(\begin{smallmatrix}
0& 1 & 1 & 1\\
1 & 0 & 0 & 0\\
0&0&0&0
\end{smallmatrix}\right) \ = \
(\{2\}, \{1\}, \{1\}, \{1\}), \\
F \ &=& \ \left(\begin{smallmatrix}
0& 1 & 1 & 0\\
1 & 0 & 0 & 0\\
0&0&0&1
\end{smallmatrix}\right) \ =
\ (\{2\}, \{1\}, \{1\}, \{3\}),\\
G \ &=& \ \left(\begin{smallmatrix}
0& 1 & 1 & 1\\
1 & 0 & 0 & 0\\
0&0&0&1
\end{smallmatrix}\right)
\ = \ (\{2\}, \{1\}, \{1\}, \{1,3\}),\;\; \\
H \ &=& \ \left(\begin{smallmatrix}
0& 1 & 1 & 1\\
1 & 1 & 0 & 0\\
0&0&0&1
\end{smallmatrix}\right)
\ = \ (\{2\}, \{1,2\}, \{1\}, \{1,3\}) \\
\end{eqnarray*}
\end{example}

\begin{figure}[h!]
\includegraphics[scale=0.4]{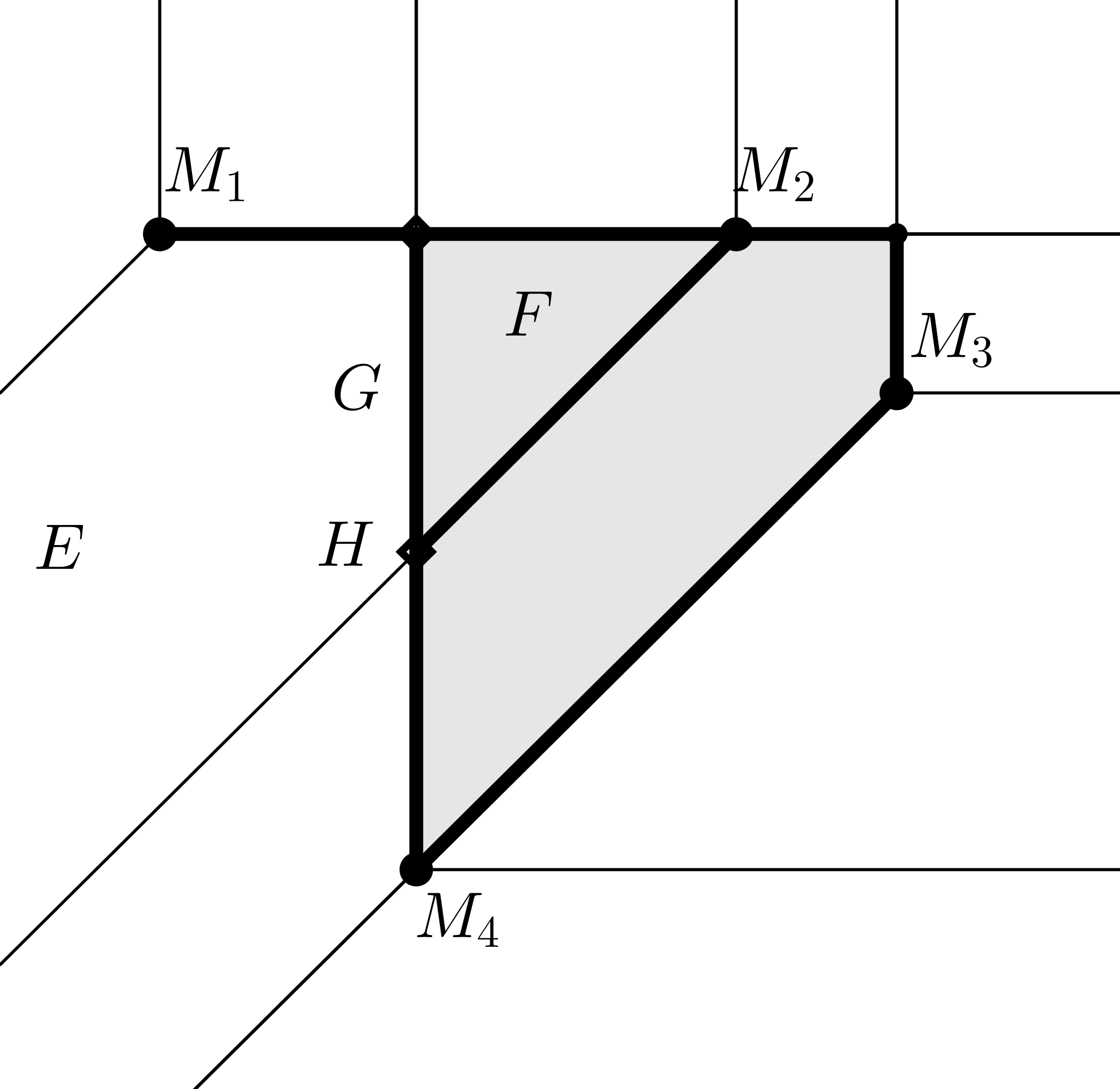}
\caption{An arrangement of four min-plus tropical hyperplanes in $\mathbb{R}^3$. The bounded cells of the resulting tropical complex are shaded. The union of the bounded regions is the max-plus polytope generated by the points $M_i$.}
\end{figure}

Let $\mathcal{F}(M)$ denote the set of all faces with respect to the
hyperplane arrangement specified by $M$. Since the faces are in one-to-one
correspondence with the types specified by this arrangement, we shall
regard $\mathcal{F}(M)$ as a subset of $\mathbb{B}^{n\times d}$.
\section{Satisfiable Boolean matrices}
Continuing with the notation of Section \ref{sect:tropfaces}, we fix a
real $n \times d$ matrix $M$ whose columns $M_{\star j}$ specify a
tropical hyperplane arrangement in $\mathbb{R}^n$; all types in this section are considered with respect to the columns of $M$. More generally, for each $S \in \mathbb{B}^{n\times d}$ we define
\begin{equation}\label{cell}D_S = \bigcap_{S_{i,j} \neq 0} {\rm Dom}_i(M_{\star j})\end{equation}
and say that $S$ is \emph{satisfiable} (with respect to $M$) if $D_S \neq \emptyset$. That is, $S$ is satisfiable if there exists $x \in \mathbb{R}^n$ such that $M_{\star j}$ dominates $x$ in position $i$ for all $i,j$ such that $S_{i,j}=1$; we say that such a point $x$ \emph{satisfies} $S$. (For example, with respect to the hyperplane arrangement given in Example 3.1 above, the (column) tuple $(\{1\}, \{2\}, \{1,2\}, \{1, 3\})$ is not satisfiable. The tuple $(\{2\}, \{2\}, \{1,2\}, \{1, 3\})$ is satisfiable, however it is not the type of any point.) Notice that if $S$ is satisfiable with respect to $M$ then the set $D_S$, which by definition is a non-empty intersection of Euclidean convex max-plus submodules of $\mathbb{R}^n$ of the form \eqref{Domset}, is itself a Euclidean convex max-plus submodule of $\mathbb{R}^n$.

It is clear from the definitions above that $x$ satisfies $S$ if and only if $S \preceq T$ where $T$ is the type of $x$, and that the cells of the tropical complex are the relative interiors of the sets $D_T$ indexed by types $T$. It follows that the set $\mathcal{S}(M)$ of all satisfiable (with respect to $M$) Boolean matrices corresponding to the hyperplane arrangement specified by $M$ is a downward-closed subset of $\mathbb{B}^{n\times d}$ and hence (by our observation at the end of Section \ref{sect:action}) $\mathcal{P}_n$ acts on the right of $\mathcal{S}(M)$. The set $\mathcal{F}(M)$ of all types with respect to $M$  is clearly not downward-closed, but it turns out that the action of $\mathcal{P}_n$ also restricts to an action on this set. In order to study this action, we begin with some foundational results concerning satisfiability which will enable us to provide a combinatorial characterization of the types.

Recall from the Introduction that we identify each partial bijection $\sigma: [d] \dashrightarrow [n]$ with the Boolean matrix $\Sigma \in \mathbb{B}^{n\times d}$ defined by $\Sigma_{i, j} =1$ if and only if $i=\sigma(j)$. We say that $S \in \mathbb{B}^{n\times d}$ \emph{contains} the partial bijection $\sigma$ if $\Sigma \preceq S$.

Our main result in this section is the following ``local'' characterization of satisfiability.

\begin{theorem}
\label{thm:bijectionsat}
Let $M \in \mathbb{R}^{n \times d}$ and $S \in \mathbb{B}^{n\times d}$. Then $S$ is satisfiable with respect to $M$ if and only if all of the partial bijections contained in $S$ are satisfiable with respect to $M$.
\end{theorem}

To prove this result we require a number of technical lemmas.

\begin{lemma}
\label{lem:sub}
Let $M \in \mathbb{R}^{n \times d}$. An element $S=(S_{1\star}, \ldots, S_{n\star})^{\mathsf{T}} \in \mathbb{B}^{n\times d}$ is satisfiable with respect to $M$ if and only if $(A, S_{2\star}, \ldots, S_{n\star})^{\mathsf{T}}$ is satisfiable with respect to $M$ for all $A \subseteq S_{1\star}$ with $|A| \leq 1$.
\end{lemma}

\begin{proof}
As we have already observed, $\mathcal{S}(M)$ is a downward-closed subset of $\mathbb{B}^{n \times d}$. It follows immediately that if $S$ is satisfiable, then so too is $(A, S_{2\star}, \ldots, S_{n\star})^{\mathsf{T}}$, where $A \subseteq S_{1\star}$.

Suppose now that $(A, S_{2\star}, \ldots, S_{n\star })^{\mathsf{T}}$ is satisfiable for each $A \subseteq S_{1\star}$ with $|A| \leq 1$. Since domination is scale invariant we note that we may assume without loss of generality (by rescaling if necessary) that each column $M_{\star s}$ has 0 in the first position. For each $s \in S_{1\star}$ choose $x(s) \in \mathbb{R}^n$ satisfying $(\{s\}, S_{2\star}, \ldots, S_{n\star})^{\mathsf{T}}$. We shall show that $x=\bigoplus_{s \in S_{1\star}} x(s)$ satisfies $S$. First note that we may choose each $x(s)$ so that its first co-ordinate (and hence the first co-ordinate of $x$) is $0$.

Since $x(s)$ satisfies $(\{s\}, S_{2\star}, \ldots, S_{n\star})^{\mathsf{T}}$, we have in particular that $M_{\star s}$ dominates $x(s)$ in position 1:
$$0=x(s)_1-M_{1,s} = {\rm min}_k\{ x(s)_k - M_{k,s}\},$$
from which it follows easily that $M_{\star s} \leq x(s)$. On the other hand, by definition of $x$ we see that $x(s)\leq x$ and hence $M_{\star s} \leq x$, and from the latter we can also deduce that $M_{\star s}$ dominates $x$ in position 1 for all $s \in S_1$:
$$x_1-M_{1,s}=0 = {\rm min}_k\{ x_k - M_{k,s}\}.$$
Now let $S'=(\emptyset, S_{2\star}, \ldots, S_{n\star})^{\mathsf{T}}$ and consider the max-plus submodule $D_{S'}$ defined via \eqref{cell}. Since each $x(s)$ satisfies $S'$, by definition it is contained in $D_{S'}$. Thus the max-plus sum $x=\bigoplus_{s \in S_{1\star}} x(s)$ must be contained in $D_{S'}$ too. Now since $M_{\star s}$ dominates $x$ in position 1 for all $s \in S_1$ and $x$ satisfies $S'$ we conclude from \eqref{cell} that $x$ satisfies $S$.
\end{proof}

It is clear that the analogues of Lemma~\ref{lem:sub} concerning positions $2, \ldots, n$ also hold.

\begin{corollary}
\label{cor:satis} Let $M \in \mathbb{R}^{n \times d}$ and $S =(S_{1\star}, \ldots, S_{n\star})^{\mathsf{T}} \in \mathbb{B}^{n\times d}$. Then $S \in \mathcal{S}(M)$ if and only if $A=(A_{1\star}, \ldots, A_{n\star})^{\mathsf{T}} \in \mathcal{S}(M)$ for all $A \subseteq S$ with $|A_{i\star}| \leq 1$.
\end{corollary}

\begin{proof}
Since $\mathcal{S}(M)$ is downward-closed, the forward implication is immediate. On the other hand, if $(A_{1\star}, \ldots, A_{n\star})^{\mathsf{T}}$ is satisfiable for all $A \preceq S$ with $|A_{i\star}| \leq 1$, then by repeated application of Lemma \ref{lem:sub} (and its analogues for positions $2, \ldots, n$) we see that $S$ must be satisfiable too.
\end{proof}

The following lemma combined with Corollary \ref{cor:satis} will enable us to prove the desired result (Theorem \ref{thm:bijectionsat}).

\begin{lemma}
\label{prop:satis}
Let $M \in \mathbb{R}^{n \times d}$ and let $S=(S_{1\star}, \ldots, S_{n\star})^{\mathsf{T}}, T=(T_{1\star}, \ldots, T_{n\star})^{\mathsf{T}} \in \mathcal{S}(M)$ be such that $i \in S_{j\star} \cap T_{k\star}$. Then there exists $U=(U_{1\star},\ldots, U_{n\star})^{\mathsf{T}}\in \mathcal{S}(M)$ such that:\\
(i) $i \in U_{j\star} \cap U_{k\star}$; and\\
(ii) if $l \in S_{p\star} \cap T_{p\star}$, then $l \in U_{p\star}$.
\end{lemma}

\begin{proof}
Let $x,y \in \mathbb{R}^n$ be such that $x$ satisfies $S$ and $y$ satisfies $T$. Thus by assumption $M_{\star i}$ dominates $x$ in position $j$ and $y$ in position $k$. We want to construct $u \in \mathbb{R}^n$ such that $M_{\star i}$ dominates $u$ in both positions $j$ and $k$, and that in position $p$, $u$ is dominated by any columns $M_{\star l}$ which dominate both $x$ and $y$ in that position.

First of all we note that we may assume without loss of generality that $M_{j,i}=M_{k,i}=0$, as translating the $i$th column of $M$ by the vector $-M_{j,i}e_j -M_{k,i}e_k$ clearly does not alter the combinatorial structure of the resulting face poset. Moreover, since domination is scale invariant, we may also assume without loss of generality that $x_j = y_k = 0$.

Recalling the definition of the residuation operator $\langle \cdot | \cdot \rangle$ given in \eqref{resid}, we define
$$u = \bigoplus_{l=1}^{d} \; {\rm min}(\langle M_{\star l} | x\rangle, \langle M_{\star l} | y\rangle)\otimes M_{\star l}$$
and claim that $u$ satisfies the required conditions. Notice that the $i$th term in the above definition of $u$ is ${\rm min(0,0)} \otimes M_{\star i} = M_{\star i}$, giving $u \geq M_{\star i}$ and hence $u_t - M_{t,i} \geq 0$ for all $t \in [n]$. It is also easy to see that $x \geq u$ and $y \geq u$, since for example by \eqref{resid} we have $x \geq \langle M_{\star l} | x\rangle \otimes M_{\star l}$ for all $l$, and hence
$$x \; \geq \; \bigoplus_{l=1}^{d}  \langle M_{\star l} | x\rangle \otimes M_{\star l} \; \geq \;  \bigoplus_{l=1}^{d}  {\rm min}(\langle M_{\star l} | x\rangle, \langle M_{\star l} | y\rangle)\otimes M_{\star l} \; = \; u.$$
Now
$$0 = M_{j,i} \leq u_j \leq x_j = 0,$$
from which we deduce that $u_j =0.$  By an identical argument using $y$ instead of $x$, we get $u_k = 0$. It now follows that
$$u_j - M_{j,i} = u_k - M_{k,i} = 0 \leq u_t - M_{t,i}, \mbox{ for all } t \in [n],$$
or in other words, that $u$ is dominated by $M_{\star i}$ in both of these positions.

Now suppose that $M_{\star l}$ dominates both $x$ and $y$ in position $p$, that is, $\langle M_{\star l} | x\rangle  = x_p - M_{p,l}$ and $\langle M_{\star l} | y\rangle = y_p - M_{p,l}$. But now looking at the $l$th term in the definition of $u$ gives
\begin{eqnarray*}
   u_p &\geq& {\rm min}(\langle M_{\star l} | x\rangle, \langle M_{\star l} | y\rangle) \otimes M_{p,l}\\
        &\geq& {\rm min}(x_p - M_{p,l},  x_p - M_{p,l}) + M_{p,l}\\
        & \geq & {\rm min}(x_p, y_p).
\end{eqnarray*}
Since we know that $u \leq x$ and $u \leq y$, giving $u_p \leq {\rm min}(x_p, y_p)$, we must have $u_p = {\rm min}(x_p, y_p)$. We are now in position to show that $M_{\star l}$ dominates $u$ in position $p$. Assume without loss of generality that $u_p=x_p \leq y_p$ (the argument being so far symmetric in $x$ and $y$). Then because $M_{\star l}$ dominates both $x$ and $y$  in position $p$ we see that
$$\langle M_{\star l}| x \rangle \;= \;x_p - M_{p,l} \; \leq \;  y_p - M_{p,l} \; = \; \langle M_{\star l}| y \rangle.$$
Thus by the definition of $u$ we see that
$$u \; \geq \; {\rm min}(\langle M_{\star l} | x\rangle, \langle M_{\star l} | y\rangle) \otimes M_{\star l} \; =\;  \langle M_{\star l} | x\rangle \otimes M_{\star l}.$$
By \eqref{resid} this means that $\langle M_{\star l} | u\rangle \geq \langle M_{\star l} | x\rangle$, giving
$${\rm min}_k\{u_k-M_{k,l}\} \; = \; \langle M_{\star l} | u\rangle \; \geq \;  \langle M_{\star l} | x\rangle \; =\;  x_p - M_{p,l} \; = \;  u_p - M_{p,l},$$
or in other words, $M_{\star l}$ dominates $u$ in position $p$.
\end{proof}

\begin{proof}[Proof of Theorem~\ref{thm:bijectionsat}]
Since $\mathcal{S}(M)$ is downward-closed, if $S$ is satisfiable, then so too is each $A\preceq S$. Suppose then that all partial bijections contained in $S$ are satisfiable. We shall show that each $A \preceq S$ with $|A_{i\star}| \leq 1$ must be satisfiable (by repeated application of Lemma~\ref{prop:satis}). Then by Corollary~\ref{cor:satis} we can conclude that $S$ is satisfiable.

Given $A \preceq S$ with $|A_{i\star}| \leq 1$ we consider the collection $\mathcal{B}_A$ of all partial bijections contained within $A$. We apply Lemma~\ref{prop:satis} to all pairs $(S,T)$ of \emph{maximal} elements $S, T \in \mathcal{B}_A$ which differ in a single column, add the resulting satisfiable element $U$ to $\mathcal{B}_A$, and continue in this fashion. We note that at each step of this procedure $\mathcal{B}_A$ contains only satisfiable elements that are less than or equal to $A$ in the order. Since there are finitely many such elements, it is clear that this process must terminate. The resulting set $\mathcal{B}_A$ then has unique maximum element $A$.
\end{proof}
\section{Satisfiable partial bijections and the max-plus permanent}
The \emph{max-plus tropical permanent} of $X \in M_k(\mathbb{R})$ is the real number
$${\rm perm}(X) = \bigoplus_{\tau \in {\rm Sym}(k)} X_{\tau(1),1} \otimes \cdots \otimes X_{\tau(k),k}$$
where ${\rm Sym}(k)$ denotes the symmetric group on $[k]$. The matrix $X$ is said to be \emph{max-plus tropically singular} if the expression for the permanent above is attained by at least two different permutations, and \emph{max-plus tropically non-singular} otherwise.

Let $\sigma : [d] \dashrightarrow [n]$ be a partial bijection with domain of definition $J \subseteq [d]$ and image $I \subseteq [n]$. We say
that $\sigma$ is \textit{permanent-attaining} (with respect to $M$) if $\sigma$ attains the maximum in the calculation of the max-plus permanent for the corresponding $k \times k$ submatrix of $M$, that is, if for every bijection $\tau : I \to J$ we have
$$\bigotimes_{j \in J}  M_{\sigma(j),j} \ \geq \ \bigotimes_{j \in J} M_{\tau(j),j}.$$
We define the \textit{permanent structure} $\mathcal{P(M)}$ of $M$ to be the set of all permanent-attaining partial bijections.

\begin{proposition}
The permanent-structure $\mathcal{P}(M)$ is a downward-closed subset of $\mathbb{B}^{n \times d}$.
\end{proposition}
\begin{proof}
Let $\sigma \in \mathcal{P}(M)$ and suppose $\tau \preceq \sigma$. Clearly $\tau$ is a partial bijection. Suppose for a contradiction
that it is not permanent-attaining. Choose a permanent-attaining partial bijection $\theta$ with the same domain of definition and image as $\tau$.
Since $\tau \preceq \sigma$, $\sigma$ maps the domain of $\theta$ onto the range of $\theta$; this means that we can extend $\theta$ to
a partial bijection $\theta'$ with the same domain and image as $\sigma$, by setting $\theta'(i)=\sigma(i)$ for all $i$ not in the domain of
$\theta$. But now it is easily seen that $\theta'$ makes a greater contribution to the permanent calculation than $\sigma$, contradicting
the assumption that $\sigma$ is permanent-attaining.
\end{proof}
It follows from our observation in Section \ref{sect:action} that $\mathcal{P}_n$ also acts on the right of $\mathcal{P}(M)$. The following theorem shows that the  permanent-attaining partial bijections are precisely the satisfiable partial bijections (or in fact, the partial bijections which are satisfied by elements of the max-plus column space of $M$).

\begin{theorem}
\label{permanentsatis}
Let $M \in \mathbb{R}^{n\times d}$ and let $\sigma: [d] \dashrightarrow [n]$ be a partial bijection. Then the following are equivalent:
\begin{itemize}
\item[(i)] $\sigma$ is permanent-attaining with respect to $M$;
\item[(ii)] $\sigma$ is satisfiable with respect to $M$;
\item[(iii)] there exists $y \in {\rm Col}_{\oplus}(M)$ satisfying $\sigma$.
\end{itemize}
\end{theorem}

\begin{proof}
Let $J \subseteq [d]$ and $I \subseteq [n]$ be the domain and image respectively of $\sigma$.

It is clear that (iii) implies (ii). We shall show that (ii) implies (i) and (i) implies (iii). Let $\Sigma=(\Sigma_{1\star}, \ldots, \Sigma_{n\star})^{\mathsf{T}} \in \mathbb{B}^{n\times d}$ be the element defined by $j \in \Sigma_{i\star}$ if and only if $\sigma(j)=i$.

Suppose that (ii) holds.  By definition there exists $y$ in
$$D_{\Sigma} = \bigcap_{j \in J} {\rm Dom}_{\sigma(j)}(M_{\star j}).$$
Thus for each $j \in J$ we have
$$y_{\sigma(j)} - M_{\sigma(j),j} = {\rm min}_k\{ y_k - M_{k,j}\},$$
or in other words
\begin{equation}
\label{yexists}
M_{k,j}- M_{\sigma(j),j} \leq y_k - y_{\sigma(j)}, \;\;\;\mbox{ for all }j \in J \mbox{ and all }k \in [n].
\end{equation}
Suppose for contradiction that $\sigma$ does not attain the permanent. Then there is another bijection $\tau:J \rightarrow I$ with $\tau \neq \sigma$ such that
$$\sum_{j \in J} M_{\sigma(j),j} < \sum_{j \in J} M_{\tau(j),j},$$
giving
$$0< \sum_{j \in J} M_{\tau(j),j} - M_{\sigma(j),j}.$$
But it follows from \eqref{yexists} that the right hand side of this inequality is less than or equal to $\sum_{j \in J} (y_{\tau(j)} - y_{\sigma(j)})$, giving a contradiction since the latter is 0 (because $\sigma$ and $\tau$ are partial bijections with the same domain and \emph{image}).

Now suppose that (i) holds. Suppose first that $|I|=|J|=n = d$. For each $\varepsilon >0$ let $M(\varepsilon)$ be the $n \times n$ square matrix such that
$$M(\varepsilon)_{i,j} = \begin{cases}
M_{i,j}&\mbox{if } i=\sigma(j)\\
M_{i,j} - \varepsilon & \mbox{otherwise}.
\end{cases}$$
Then the permanent of $M(\varepsilon)$ is attained uniquely by $\sigma$ and hence $M(\varepsilon)$ tropically non-singular. It follows from \cite[Theorem 4.2]{DevSanStu05} and \cite[Proposition 17]{DevStu04} that there is a vector $y(\varepsilon) \in {\rm Col}_{\oplus}(M(\varepsilon))$ whose type \emph{with respect to the columns of $M(\varepsilon)$} is an $n$-tuple of singleton sets whose union is $[n]$. Moreover, it is not hard to see that ${\rm type}(y(\varepsilon))$ - again, with respect to the columns of $M(\varepsilon)$ -  must be $\Sigma$, since (having already shown that (ii) implies (i)) we note that any other vector of singletons will yield another permanent-attaining permutation.

For each $\varepsilon>0$ we define
$$D_\Sigma(\varepsilon) = \bigcap_{j \in J} {\rm Dom}_{\sigma(j)}(M(\varepsilon)_{\star j}).$$
Thus $D_\Sigma(\varepsilon)$ is the set of all elements $x \in \mathbb{R}^n$ having $\Sigma \preceq {\rm type}(x)$. Since $\Sigma$ is the type of a point in the max-plus column space of $M(\varepsilon)$, we note that by \cite[Theorem 15]{DevStu04} $D_\Sigma(\varepsilon)$ corresponds to a closed and bounded subset of $\mathbb{R}^{n-1}$ (identified with tropical projective space $\mathbb{R}^n/\mathbb{R}(1,\ldots, 1)$ as in \eqref{eq:projection}). Moreover, it is easy to see that as $\varepsilon \rightarrow 0$ we obtain a decreasing sequence of nested closed and bounded subsets of $\mathbb{R}^{n-1}$. Fixing $\kappa>0$, the set
$$D_\Sigma(\kappa) = \bigcup_{0 < \varepsilon \leq \kappa} D_{\Sigma}(\varepsilon)$$
is therefore compact in projective space, with respect to the usual topology induced from $\mathbb{R}^{n-1}$. Since each $y(\varepsilon)$ with $\varepsilon \leq \kappa$ is contained in $D_\Sigma(\kappa)$, there is a convergent subsequence of $y(\varepsilon)$'s in projective space, the limit of which must satisfy the limiting set of inequalities. In other words, there exists $y \in \mathbb{R}^n$ whose type $T$ \emph{with respect to the columns of $M$} contains $\Sigma$. By \cite[Corollary 12]{DevStu04} we note that face corresponding to $T$ in the tropical complex generated by $M$ must be bounded (each row of $T \in \mathbb{B}^{n \times d}$ is non-zero) and hence, by \cite[Theorem 15]{DevStu04}, this face is contained in the column space of $M$.  But now $y$ satisfies $\Sigma$ with respect to $M$ and $y \in {\rm Col}_{\oplus}(M)$.

Finally, suppose that $1 \leq k=|I|=|J| \leq n, d$. Let $X$ denote the $k \times k$ submatrix of $M$ whose rows are indexed by $I$ and whose columns are indexed by $J$. By the argument above, there exists $\hat{y}$ in the column space of $X$ satisfying $\sigma$, that is:
$$\hat{y} \in {\rm Col}_{\oplus}(X) \cap \bigcap_{j \in J} {\rm Dom}_{\sigma(j)} (X_{\star j}) \subseteq \mathbb{R}^k.$$
Since $\hat{y} \in {\rm Col}_{\oplus}(X)$ we can write $\hat{y} = \bigoplus_{j \in J} \alpha_j \otimes X_j$ for some $\alpha_j \in \mathbb{R}$, and one can deduce from definition \eqref{resid} that
$$\hat{y} = \bigoplus_{j \in J} \langle X_{\star j} | \hat{y} \rangle \otimes X_{\star j}.$$
Now set
$$y = \bigoplus_{j \in J} \langle X_{\star j} | \hat{y} \rangle \otimes M_{\star j}.$$
Thus $y \in {\rm Col}_{\oplus}(M)$ and $y_t = \hat{y}_t$ for all $t \in I$. Notice that for all $j \in J$
$$ \langle X_j | \hat{y} \rangle = \hat{y}_{\sigma(j)} - X_{\sigma(j),j} = y_{\sigma(j)} - M_{\sigma(j),j}.$$
Thus for all $t \in [n]$ and all $j \in J$ we have
$$y_t ={\rm max}_{l \in J}(\langle X_{\star l} | \hat{y}\rangle + M_{t,l}) \geq \langle X_{\star j} | \hat{y} \rangle + M_{t,j} = y_{\sigma(j)} - M_{\sigma(j),j} + M_{t,j},$$
giving $y_t - M_{t,j} \geq y_{\sigma(j)} - M_{\sigma(j),j}$
for all $j \in J$ and all $t \in [n]$. In other words, $M_j$ dominates $y$ in position $\sigma(j)$ for all $j \in J$, showing that $y$ satisfies $\Sigma$ as required.
\end{proof}

In light of the previous theorem we may rephrase Theorem \ref{thm:bijectionsat} as follows:

\begin{corollary}
\label{cor:bijectionsat}
Let $M \in \mathbb{R}^{n \times d}$ and let $S \in \mathbb{B}^{n\times d}$. Then $S$ is satisfiable with respect to $M$ if and only if every partial bijection contained in $S$ is permanent-attaining with respect to $M$.
\end{corollary}

It follows from Theorem \ref{permanentsatis} that one may deduce the permanent structure $\mathcal{P}(M)$ from the collection of all types $\mathcal{F}(M)$. In the next section we shall prove the converse, hence showing that $\mathcal{P}(M)$ and $\mathcal{F}(M)$ encode exactly the same combinatorial data about the tropical complex. We conclude this section with the following crucial observation.

\begin{lemma}
\label{IJtypes}
Let $M \in \mathbb{R}^{n \times d}$ and let $T \in \mathbb{B}^{n \times d}$ be a type with respect to the columns of $M$. If $T$ contains a partial bijection  $\sigma: [d] \dashrightarrow [n]$, then $T$ contains all of the permanent-attaining (with respect to $M$) partial bijections that have the same domain of definition and image as $\sigma$.
\end{lemma}

\begin{proof}
Since $T$ is a type, there exists $x \in \mathbb{R}^n$ having type $T$. By definition, $x$ satisfies $A \in \mathbb{B}^{n \times d}$ only if $A\preceq T$, and so in particular $x$ satisfies $\sigma$. Now if $\tau$ is any permanent-attaining partial bijection with the same domain of definition and image as $\sigma$ it follows from Theorem \ref{permanentsatis} that $\tau$ is satisfiable. We show that $x$ also satisfies $\tau$, from which we conclude that $T$ contains $\tau$.

Since $\sigma$ and $\tau$ are satisfiable it follows from Theorem \ref{permanentsatis} that these partial bijections attain the permanent $P$ in the submatrix of $M$ with rows indexed by $I$ and columns indexed by $J$:

$$P \ = \ \sum_{j \in J} M_{\sigma(j),j} \ = \ \sum_{j \in J} M_{\tau(j),j}.$$

Suppose that $x$ satisfies $\sigma$. In other words, suppose that
$$x_{\sigma(j)} - M_{\sigma(j),j} \leq x_t - M_{t,j} \mbox{ for all }j \in J \mbox{  and all }t \in [n].$$
We note that since the image of $\tau$ is a subset of $[n]$, we have in particular
$$x_{\sigma(j)} - M_{\sigma(j),j} \leq x_{\tau(j)} - M_{\tau(j),j} \mbox{ for all }j \in J.$$
We claim that each of the last inequalities has to be an equality. Indeed, suppose not. Then summing them gives
$$\sum_{j \in J} x_{\sigma(j)} - M_{\sigma(j),j} \ < \ \sum_{j \in J} x_{\tau(j)} - M_{\tau(j),j}.$$
But since $\sigma$ and $\tau$ are bijections with the same \emph{image} $I$, we see that this becomes
$$\sum_{i \in I} x_{i} - \sum_{j \in J}M_{\sigma(j),j} \ < \ \sum_{i \in I} x_{i} - \sum_{j \in J}M_{\tau(j),j}.$$
We note that the first summations on either side are identical, whilst each of the remaining summations is equal to $P$ by assumption, giving a contradiction and proving the claim. Thus we must have
$$x_{\tau(j)} - M_{\tau(j),j} = x_{\sigma(j)} - M_{\sigma(j),j} \leq x_{t} - M_{t,j}$$
for all $j \in J$ and all $t \in [n]$. Or in other words, $x$ satisfies $\tau$.
\end{proof}

\section{Characterizing types}

Given a matrix $M$ defining a tropical complex and an $n \times d$ Boolean matrix (or equivalently, a $d$-tuple of subsets of $[n]$) $S$, it is natural to ask whether $S$ is a type of the complex. Our results so far suffice to establish some elementary properties which $S$ must
possess if it is to be a type; our aim in this section is to establish a collection of conditions which are necessary and sufficient for $S$
to be a type:
\begin{theorem}
\label{thm:types}
Let $M \in \mathbb{R}^{n \times d}$ and $S \in \mathbb{B}^{n\times d}$. Then $S$ is a type of the tropical complex corresponding to $M$ if and only if:
\begin{itemize}
\item[(T1)] every column of $S$ is non-empty;
\item[(T2)] every partial bijection contained in $S$ is permanent-attaining with respect to $M$; and
\item[(T3)] if $S$ contains a partial bijection $\sigma$, then $S$ contains all of the permanent-attaining (with respect to $M$) partial bijections with the same domain of definition and image as $\sigma$.
\end{itemize}
\end{theorem}

Notice that the three conditions (T1), (T2) and (T3) depend only on the matrix $S$ and the permanent structure of $M$; there is no direct dependence on the matrix $M$ or its
associated tropical complex. It follows that one can determine whether a given Boolean matrix is a type by reference only to the permanent structure. Conversely, by
Theorem~\ref{permanentsatis} a partial bijection is permanent-attaining if and only if it is contained in a type, so the permanent structure can be also be recovered from the
 types. This justifies the following claim which we made earlier:
\begin{corollary}\label{cor:sameinfo}
The types of the tropical complex associated to a matrix $M$ are completely determined by the permanent structure of $M$, and vice versa.
\end{corollary}

We now turn our attention to the proof of Theorem~\ref{thm:types}. We begin with the direct implication, which is straightforward modulo what we have already shown:

\begin{proof}[Proof of the direct implication of Theorem~\ref{thm:types}]
Suppose $S$ is a type. Choose a point $x \in \mathbb{R}^n$ of which it is the type. Then:
\begin{itemize}
\item[(T1)] Since the point $x$ must lie in at least one of the sectors relative to the hyperplane specified by a column of $M$, we see that each column of $M$ dominates $x$ in some position and thus each column of $S$ must be non-empty.
\item[(T2)] The point $x$ satisfies every $T$ such that $T \preceq S$, and hence in particular every partial bijection contained in $S$. Thus, by Theorem~\ref{permanentsatis},
every such partial bijection must be permanent-attaining.
\item[(T3)] This is Lemma \ref{IJtypes} above.
\end{itemize}
\end{proof}

To establish the converse implication of Theorem~\ref{thm:types} we shall need some further technical lemmas:
\begin{lemma}
\label{technical}
Let $S=(S_{1\star}, \ldots, S_{n\star})^{\mathsf{T}} \in \mathcal{S}(M)$ satisfying conditions (T1) and (T3) above and choose a point $x$ satisfying $S$. Let $T=(T_{1\star}, \ldots, T_{n\star})^{\mathsf{T}}$ be the type of $x$ and define a binary relation $\sim$ on $[n]$ by $p \sim q$ if $S_{p\star} \cap T_{q\star} \neq \emptyset$.

\begin{itemize}
\item[(i)] If $p \sim q$ then $x_{p} - x_{q} = M_{p,g} - M_{q,g}$ for all $g \in S_{p\star} \cap T_{q\star}$.
\item[(ii)] If $p \sim q$, $y$ satisfies $S$ and $y_p = x_p + \varepsilon$ for some $\varepsilon>0$, then $y_q \geq x_q + \varepsilon$.
\item[(iii)] Suppose that $I \subseteq [n]$ with the property that if $p \in I$ and $p \sim q$ then $q \in I$. Let $\varepsilon >0$ and let $y(\varepsilon)$ be the element obtained from $x$ by adding $\varepsilon$ to the entries indexed by $I$. Then there exists $\kappa>0$ such that $y(\varepsilon)$ satisfies $S$ for all $0 <\varepsilon<\kappa$.
\end{itemize}
\end{lemma}

\begin{proof}
(i) Suppose that $p \sim q$ and let $g \in S_{p\star} \cap T_{q\star}$. Since $g \in S_{p\star}$ and $x$ satisfies $S$, we have $x_{p} - M_{p,g} \leq x_q - M_{q,g}.$ On the other hand, since $g \in T_{q\star}$ and $x$ satisfies $T$, we have $x_{q} - M_{q,g} \leq x_p - M_{p,}.$ It follows immediately from these two inequalities that $x_{p} - x_{q} = (M_g)_p - (M_g)_q.$

(ii) Suppose that $p \sim q$ and choose $g \in S_{p\star} \cap T_{q\star}$. By part (i) we know that $x_{p} - x_{q} = M_{p,g} - M_{q,g}$. Now, since $y$ satisfies $S$ we have
$$y_{p} - M_{p,g} \leq y_q - M_{q,g}.$$
Rearranging this last inequality and using the fact that $y_p = x_p + \varepsilon$ gives
$$x_p + \varepsilon = y_{p} \leq y_q+ M_{p,g}  - M_{q,g} = y_q + x_p - x_q,$$
and hence $x_q + \varepsilon \leq y_q$ as required.

(iii) For any $\varepsilon >0$ we have that $x \leq y(\varepsilon)$ with equality in positions not indexed by $I$. It follows that $y(\varepsilon)$ must satisfy the conditions specified by the subsets $S_{k\star}$ with $k \notin I$; specifically, if $j \in S_{k\star}$ for some $k \notin I$, then
$$y(\varepsilon)_k - M_{k,j} = x_k - M_{k,j} \leq x_t - M_{t,j} \leq y(\varepsilon)_t - M_{t,j}, \mbox{ for all } t \in [n].$$
Moreover, $j \in S_{k\star}$ for some $k \in I$, then
$$y(\varepsilon)_k - M_{k,j} = x_k + \varepsilon - M_{k,j} \leq x_t + \varepsilon - M_{t,j}, \mbox{ for all } t \in [n].$$
In particular, this gives
$$y(\varepsilon)_k - M_{k,j} \leq y(\varepsilon)_t - M_{t,j}, \mbox{ for all } t \in I.$$
It remains to show that we can choose $\varepsilon$ small enough so that $y(\varepsilon)$ satisfies the remaining inequalities of $S$ too. Thus we need to be able to choose $\varepsilon>0$ to simultaneously satisfy the inequalities
$$x_k + \varepsilon -M_{k,j} \leq x_q - M_{q,j}, \mbox{ where } k \in I, j \in S_{k\star} \mbox{ and } q \notin I.$$
Or in other words, we want to show that there exists $\varepsilon$ satisfying
$$0< \varepsilon \leq {\rm min}\{ x_q-x_k +M_{k,j} - M_{q,j}\},$$
where the minimum takes place over all $k \in I$, $j \in S_{k\star}$ and $q \notin I$. Suppose not.  Then, since there are only finitely many terms in the minimum, we must have
$$x_q-x_k +M_{k,j} - M_{q,j} \leq 0 \; \mbox{ for some } k \in I, j \in S_{k\star} \mbox{  and }q \notin I.$$
Rearranging this gives
$$x_q - M_{q,j} \leq x_k -M_{k,j}.$$
Since $j \in S_{k\star}$ we know that $x$ is dominated by $M_{\star j}$ in position $k$, and so
$$x_k - M_{k,j} = {\rm min}_t\{x_t -M_{t,j}\} \leq x_q -M_{q,j}.$$
But then the last two inequalities combine to give
$$x_q - M_{q,j} = x_k - M_{k,j} = {\rm min}_t\{x_t -M_{t,j}\},$$
showing that $j \in S_{k\star} \cap T_{q\star}$, and hence $k \sim q$. Since $k \in I$, by assumption we must have that $q \in I$ too, contradicting $ q \notin I$.
\end{proof}

We shall use the previous lemma to give an inductive argument for our characterization of types.

\begin{lemma}
\label{induct}
Let $S \in \mathcal{S}(M)$ satisfying conditions (T1) and (T3) above, and choose a point $x \in \mathbb{R}^n$ satisfying $S$. Let $T$ be the type of $x$ and suppose that $T \neq S$. Then there exists $y \in \mathbb{R}^n$ such that $S \preceq {\rm type}(y) \preceq T$ with ${\rm type}(y) \neq T$.
\end{lemma}

\begin{proof}
We shall construct $y$ by increasing $x_i$ by a small amount $\varepsilon$ for all positions $i$ contained in some carefully chosen subset $I \subseteq [n]$. It is straightforward to check that given any choice of $I$ we can choose $\varepsilon$ small enough so that ${\rm type}(y)\preceq T$.

Suppose that $I$ has the property that if $p \in I$ and $p \sim q$ then $q \in I$. Then Lemma~\ref{technical} (iii) tells us that we can choose $\varepsilon$ so that $S$ is still satisfied. Thus for any such a choice of $I$ we may choose $\varepsilon$ small enough so that $S \preceq {\rm type}(y) \preceq T$. It remains to show that some choice of $I$ gives ${\rm type}(y)\neq T$.

Let $S = (S_{1\star}, \ldots, S_{n\star})^{\mathsf{T}}$ and $T= (T_{1\star}, \ldots, T_{n\star})^{\mathsf{T}}$. Since $S \neq T$, we have $j \in T_{i\star}$, but $j \notin S_{i\star}$ for some $j \in [d]$ and $i \in [n]$. Since $S$ satisfies (T1), $j$ must occur somewhere in $S$. Suppose that $j \in S_{m\star}$. We first show that choosing $I$ so that $i \in I$ and $m \notin I$ yields $j \notin {\rm type}(y)_{i \star}$ and hence ${\rm type}(y)\neq T$. Suppose for contradiction that $j \in {\rm type}(y)_{i\star}$. Then $M_{\star j}$ dominates $y$ in position $i$ we have in particular
$$y_i - M_{i,j} \leq y_m - M_{m,j}$$
and since $i \in I$ and $m \notin I$ this gives
$$x_i + \varepsilon  - M_{i,j} \leq x_m - M_{m,j}.$$
But since $j \in S_{m\star}$ and $x$ satisfies $S$, we also know that $M_{\star j}$ dominates $x$ in position $m$, giving in particular that
$$x_m - M_{m,j} \leq x_i - M_{i,j}.$$
It is easy to see that the last two inequalities combined contradict $\varepsilon >0$.

To complete the proof we claim that there is a subset $I \subseteq [n]$ such that
\begin{itemize}
\item[(a)] if $p \in I$ and $p \sim q$ then $q \in I$; and
\item[(b)] $i \in I$ and $m \notin I$.
\end{itemize}
Suppose for contradiction that any subset $I$ satisfying condition (a) and containing $i$ must also contain $m$. It follows that there must be a sequence
$$i \sim p_2 \sim \cdots \sim p_{r-1} \sim m.$$
By relabelling co-ordinates as necessary we may assume that $i=1$, $p_2=2, \ldots, p_{r-1}=r-1$ and $m=r$, so that $j \in T_{1\star}\cap S_{r\star}$, but $j \notin S_{1\star}$ and
$$1 \sim 2 \sim \cdots \sim r-1 \sim r.$$
Thus for $s=1, \ldots, r-1$ we may choose $j_s \in S_{s\star} \cap T_{{s+1}\star}$, giving:
\begin{eqnarray*}
Q=(j_1, j_2, \ldots, j_{r-1}, j, \emptyset, \ldots, \emptyset)^{\mathsf{T}} \preceq S\\
R=(j ,j_1, j_2, \ldots, j_{r-1}, \emptyset, \ldots, \emptyset)^{\mathsf{T}} \preceq T.
\end{eqnarray*}
It is clear that any partial bijection contained in $Q$ or $R$ will be satisfiable (by Theorem \ref{thm:bijectionsat} it will be satisfied by $x$), and hence must attain the permanent in the corresponding submatrix of $M$. We claim that there are partial bijections $Q' \preceq Q$ and $R' \preceq R$ such that $Q'$ and $R'$ have the same domain of definition and the same image, with $Q'_1=Q_1$ and $R'_1=R_1$. We first note that this will give the desired contradiction.  Since $S$ satisfies condition (T3) we know that the set of all partial bijections contained in $S$ contains either all or none of the permanent-attaining bijections on any given submatrix. Thus it follows from the fact that $Q'$ is a partial bijection contained in $S$ that $R'$ (a partial bijection with the same domain and image) must be contained in $S$ too. But then $\{j\} = R'_1 \subseteq S_1$, giving a contradiction.

The partial bijections $Q'$ and $R'$ are constructed as follows. We know that $j_1 \in Q_{1\star}$, $j \in R_{1\star}$, $j_1 \neq j$ and $j$ occurs in at least one row of $Q$. Let $i_1 \in \{2, \ldots, r\}$ be the first row of $Q$ in which $j$ occurs and let $\rho(j)$ denote the element contained in $R_{i_1\star}$. Since $R$ is just $Q$ shifted one place to the right it is clear that $\rho(j) \neq j$. If $\rho(j)=j_1$, then we will have partial bijections of the form
$$\begin{array}{c c c c c c c c c c c}
Q' &=&(j_1, & \emptyset, & \ldots, &\emptyset,& j,& \emptyset,& \ldots,& \emptyset)^{\mathsf{T}}& \preceq S\\
R'&=& (j ,&\emptyset,& \ldots,& \emptyset,& j_{1},& \emptyset,& \ldots,& \emptyset)^{\mathsf{T}}& \preceq T.
\end{array}
$$

Otherwise, $\rho(j)\neq j_1$. Let $i_2$ be the first position of $Q$ in which $\rho(j)$ occurs. Since $R$ shifts $Q$ one place to the right it is clear that $i_2<i_1$. Let $\rho^2(j)$ denote the element contained in $R_{i_2}$. Again, since $R$ shifts $Q$ one place to the right it is clear that $\rho^2(j) \neq \rho(j)$. If $\rho^2(j)=j_1$, then we will have partial bijections of the form
$$\begin{array}{c c c c c c c c c c c c c c c}
Q' &=&(j_1, & \emptyset, & \ldots, &\emptyset,&\rho(j)& \emptyset, & \ldots, &\emptyset,&  j,& \emptyset,& \ldots,& \emptyset)^{\mathsf{T}}& \preceq S\\
R'&=& (j ,&\emptyset,& \ldots,& \emptyset,&j_1& \emptyset, & \ldots, &\emptyset,& \rho(j),& \emptyset,& \ldots,& \emptyset)^{\mathsf{T}}& \preceq T.
\end{array}
$$
Otherwise $\rho^2(j) \neq \rho(j), j_1$. Continuing in this way will yield a cycle of the required form.
\end{proof}

We are now ready to complete the proof of Theorem~\ref{thm:types}:
\begin{proof}[Proof of the converse implication of Theorem~\ref{thm:types}]
Suppose then that $S$ satisfies the conditions (T1)-(T3). Since $S$ satisfies (T2), it follows from Corollary \ref{cor:bijectionsat} that we may choose $x\in \mathbb{R}^n$ such that $x$ satisfies $S$. Choose $x$ so that the type of $x$ is minimal with respect to the partial order $\preceq$, and denote this type by $T$. Thus $S \preceq T$ and if $T=S$, then we are done. Otherwise, we may apply Lemma \ref{induct} to find $y\in \mathbb{R}^n$ such that $S\preceq {\rm type}(y) \preceq T$ with ${\rm type}(y)\neq T$, contradicting our choice of $x$ with $T$ minimal.
\end{proof}

Our characterization of types gives a combinatorial way to understand the action of $\mathcal{P}_n$ on $\mathcal{F}(M)$ via the action on the set $\mathcal{P}(M)$.

\begin{corollary}
The right action of the hyperplane face monoid $\mathcal{P}_n$ of the braid arrangement on $\mathbb{B}^{n\times d}$ restricts to an action on $\mathcal{F}(M)$.
\end{corollary}

\begin{proof}
Let $T=(T_{\star 1}, \dots, T_{\star d}) \in \mathcal{F}(M)$ and $P \in \mathcal{P}_n$. Using our characterization of types, we show that $T \circ P \in \mathcal{F}(M)$, by
showing that it satisfies the three conditions given by Theorem~\ref{thm:types}.

(T1) By definition we see that the $j$th column of $T \circ P$ is $T_{\star j} \circ P$. Since $T$ is a type, we know that $\emptyset \neq T_{\star j} \subseteq [n]$ and since $P$ is an ordered set partition of $[n]$ it is therefore clear that $T_{\star j}$ has non-empty intersection with at least one component of $P$. Thus by \eqref{setaction} we conclude that $T_{\star j} \circ P$ is non-empty.

(T2) By Proposition~\ref{prop_paction} we have $T \circ P \preceq T$, so any partial bijection contained in $T \circ P$ must also be contained in $T$, and so by Theorem~\ref{thm:types}
is permanent-attaining.

(T3) Suppose that $\sigma$ and $\tau$ are partial bijections contained in $T$ with the same domain $J \subseteq [d]$ and image $I \subseteq [n]$. By our previous remarks, it suffices to show that if $\sigma$ is contained in $T \circ P$ then $\tau$ is contained in $T\circ P$ too. Assume by relabelling if necessary that $I =J=[k]$, and that $\sigma(j) =j$ for all $i \in [k]$, with $\tau$ a non-identity permutation of $[k]$. Since $\sigma$ (the identity map on $[k]$) and $\tau$ are both contained in $T$, by definition we know that:
\begin{equation}
\label{eachj}j, \tau(j) \in T_{\star j}\; \mbox{ for all }j=1, \ldots, k.
\end{equation}
For each $j\in [k]$ we define
\begin{equation}
\label{indexm}
m(j) := {\rm{max}}\{m: T_{\star j} \cap P_m \neq \emptyset\},
\end{equation}
so that by \eqref{setaction} we have $T_{\star j} \circ P = P_{m(j)}$. Since $\sigma$ (the identity) is contained in $T \circ P$ it follows that
\begin{equation}
\label{partm}
j \in P_{\star m(j)} \; \mbox{ for all }j\in[k].
\end{equation}

Now fix $j$ and look at the orbit, $\{\tau(j), \ldots, \tau^p(j)=j\}$ say, of $j$ under $\tau$. By applying \eqref{eachj} and \eqref{partm} to each of the elements $\tau^i(j) \in [k]$ in turn, we obtain:
\begin{eqnarray*}
\tau(j) &\in& T_{\star \tau(j)}, T_{\star j}, P_{\star m(\tau(j))}\\
\tau^2(j) &\in& T_{\star \tau^2(j)}, T_{\star \tau(j)}, P_{\star m(\tau^2(j))}\\
\vdots\\
\tau^{p-1}(j) &\in& T_{\star \tau^{p-1}(j)}, T_{\star \tau^{p-2}(j)}, P_{\star m(\tau^{p-1}(j))}\\
j=\tau^p(j) &\in& T_{\star j}, T_{\star \tau^{p-1}(j)}, P_{\star m(j)}\\
\end{eqnarray*}
Note that if $m(j) < m(\tau(j))$ then by definition \eqref{indexm} we see that $T_{\star j} \cap P_{\star m(\tau(j))} = \emptyset$. Since $\tau(j)$ is contained in this intersection, we must have $m(j) \geq m(\tau(j))$. Continuing in this way we note that
$$m(j) \geq m(\tau(j)) \geq \cdots \geq m(\tau^p(j)) = m(j),$$ from which we conclude that these indices all coincide and hence by combining the information above we see that the set $P_{\star m(j)}$ contains the entire orbit of $j$ under $\tau$. In particular, $\tau(j) \in P_{\star m(j)}$ for all $j=1,\ldots, k$, hence showing that $\tau$ is contained in $T \circ P$ as required.
\end{proof}

Since $A\circ P \preceq A$ for all $A \in \mathbb{B}^{n \times d}$ it follows that if a type $T$ is equal to $T' \circ P$ for some other type $T'$, then $T'$ must be a face of $T$. Moreover, it is easy to see that each maximal dimensional cell contained in $\mathcal{F}(M)$ is a fixed point of the right action of $\mathcal{P}_n$. Thus (just as in the case of usual hyperplane arrangements), this right action cannot be used to generate a random walk on the maximal dimensional cells (``chambers'') of the tropical hyperplane arrangement. We also note that although $\mathcal{P}_n$ acts upon $\mathcal{F}(M)$, this action does not immediately restrict to give an action on the max-plus polytope generated by the columns of $M$ (that is, the collection of bounded cells). For example, acting upon the face $H$ of the tropical hyperplane arrangement given in  Example 3.1 by the element $(\{3\}, \{2\}, \{1\}) \in \mathcal{P}_3$ gives the unbounded face $E$, which is not contained in the max-plus column space of $M$.

\bibliographystyle{plain}

\end{document}